\documentclass[12pt]{amsart}
\usepackage{amsmath,amsthm,amssymb,euscript,verbatim, enumerate}
\newtheorem{theorem}{Theorem}
\newtheorem{prop}[theorem]{Proposition}
\newtheorem{lemma}[theorem]{Lemma}
\newtheorem{corollary}{Corollary}

\theoremstyle{definition}
\newtheorem{defn}{Definition}
\newtheorem{remark}{Remark}

\newcommand{\F}{{\EuScript F}}
\newcommand{\X}{{\EuScript X}}

\newcommand{\R}{\mathbb R}
\newcommand{\C}{\mathbb C}
\newcommand{\CP}{\C\mathrm{P}}
\newcommand{\CH}{\C{\mathrm H}}
\newcommand{\RH}{\R{\mathrm H}}
\newcommand{\U}{\EuScript U}
\newcommand{\W}{\EuScript W}
\def\({\left(}
\def\){\right)}
\def\<{\left <}
\def\>{\right >}
\def\a {\alpha}

\def\l {\lambda}
\def\wt{\widetilde\omega}
\def\&{\wedge}

\newcommand{\I}{{\mathcal I}}
\newcommand{\J}{{\EuScript J}}
\newcommand{\Jhat}{\widehat{\J}}
\newcommand{\K}{{\mathcal K}}
\newcommand{\Khat}{\widehat{\K}}

\renewcommand{\H}{\mathcal H}

\newcommand{\w}{\omega}

\newcommand{\bv}{\mathbf v}

\newcommand{\bz}{\mathbf z}

\newcommand{\ri}{\mathrm i}
\renewcommand{\Re}{\operatorname{Re}}
\renewcommand{\Im}{\operatorname{Im}}

\def\intprod{\mathbin{\raisebox{.4ex}{\hbox{\vrule height .5pt width
5pt depth 0pt %
         \vrule height 3pt width .5pt depth 0pt}}}}

\begin{document}
\title{Hypersurfaces in $\CP^2$ and $\CH^2$ with two distinct principal curvatures}
\author{Thomas A. Ivey}
\address{Dept. of Mathematics, College of Charleston\\
66 George St., Charleston SC 
}
\email{IveyT@cofc.edu}
\author{Patrick J. Ryan}
\address{Department of Mathematics and Statistics\\
              McMaster University\\
              Hamilton, ON  Canada
}
\email{ryanpj@mcmaster.ca}
\date{\today}
\begin{abstract}
It is known that hypersurfaces in $\CP^n$ or $\CH^n$ for which the number $g$ of distinct principal curvatures
satisfies $g \le 2$, must belong to a standard list of Hopf hypersurfaces with constant principal curvatures, provided
that $n \ge 3$.  In this paper, we construct a 2-parameter family of non-Hopf hypersurfaces in $\CP^2$ and $\CH^2$ with $g=2$
and show that every non-Hopf hypersurface with $g=2$ is locally of this form.
\end{abstract}
\maketitle

\section{Introduction}\label{intro}

It is known that hypersurfaces in $\CP^n$ or $\CH^n$ for which the number $g$ of distinct principal curvatures
satisfies $g \le 2$ must be members of the Takagi/Montiel lists of homogeneous Hopf hypersurfaces,
provided that $n \ge 3$.  (See Theorems 4.6 and 4.7
of \cite{nrsurvey}).  In particular, they must be Hopf.  In this paper, we investigate the case $n=2$.

We first show in Theorem \ref{theorem1} that Hopf hypersurfaces
in $\CP^2$ or $\CH^2$ with $g\le 2$ must be in the Takagi/Montiel lists.
However, it turns out that there are also non-Hopf examples with $g \le 2$ and the rest of the paper will be devoted to studying them.
\begin{remark}
After the completion of this work we have learned of a preprint by D\'iaz-Ramos, Dom\'inguez-Vazquez and Vidal-Casti\~neira
\cite{DDV} where they classify hypersurfaces with two principal curvatures in $\CP^2$ and $\CH^2$ using the notion of
polar actions.
\end{remark}
In what follows, all manifolds are assumed connected and all manifolds and maps are assumed smooth ($C^\infty$)
unless stated otherwise.

The authors are grateful to the referee for reading the paper in great detail and making several valuable observations and suggestions which led to its improvement.

\section{Basic equations and results for hypersurfaces}
We follow the notation and terminology of \cite{nrsurvey}.  $M^{2n-1}$ will be a hypersurface in a complex space form, either $\CP^n$ or $\CH^n$, of constant holomorphic sectional curvature $4c = \pm 4/r^2$.  The locally defined field of unit normals is $\xi$, the structure vector field is $W = -J\xi$ and $\varphi$ is the tangential projection of the complex structure $J$.  The holomorphic distribution consisting of all tangent vectors orthogonal to $W$ is denoted by $W^\perp$ and $\varphi^2 \bv = -\bv$ for all $\bv \in W^\perp$.

The shape operator $A$ of $M$ is defined by
$$
A \bv = -\widetilde\nabla_{\bv} \xi
$$
where $\widetilde\nabla$ is the Levi-Civita connection of the ambient space and $\bv$ is any tangent vector to $M$.
(It follows that $\varphi A X = \nabla_X W$ for any tangent vector $X$.)
The eigenvalues of $A$ are the principal curvatures and the corresponding eigenvectors and eigenspaces are said to be principal vectors and principal spaces.  The  function $\<A W, W\>$ is denoted by $\a$.  If $W$ is a principal vector at all points of
$M$ (and so $AW = \a W$), we say that $M$ is a {\it Hopf hypersurface} and $\a$ is called the Hopf principal curvature.  For a Hopf hypersurface, the Hopf principal curvature is constant.  We state
the following fundamental facts (see Corollary 2.3 of \cite{nrsurvey}).
\begin{lemma}\label{fundamental}
Let $M$ be a Hopf hypersurface and let $X \in W^\perp$ be a principal vector with associated principal curvature $\lambda$.
Then
\begin{enumerate}
\item
$(\lambda - \frac{\a}{2}) A \varphi X = (\frac{\lambda \a}{2}+ c)\varphi X$;
\item
If $A \varphi X = \nu \varphi X$ for some scalar $\nu$, then $\lambda\nu = \frac{\lambda+\nu}{2}\a + c$;
\item
If $\nu = \l$ in (2), then $\nu^2 = \a\nu + c$.
\end{enumerate}
\end{lemma}

\subsection{Takagi's list and Montiel's list}

There is a distinguished class of model hypersurfaces, which we list below.  We use the established
nomenclature (types $A, B, C, D, E$ with subdivisions $A_0$, $A_1$, etc.) due to Takagi \cite {takagi}
and Montiel \cite{montiel}.
These lists consist precisely of the complete Hopf hypersurfaces with constant principal
curvatures in their respective ambient spaces as determined by Kimura \cite{kimura} and Berndt
\cite{Berndt}.
Equivalently, it is the list of homogeneous Hopf hypersurfaces, a fact which follows from the work of
Takagi \cite{takagi0} and Berndt \cite{Berndt}.  Non-Hopf homogeneous hypersurfaces exist in $\CH^n$ but not in $\CP^n$.

{\bf Takagi's list for $\CP^n$}
\begin{itemize}
\item {($A_1$)} Geodesic spheres (which are also tubes over totally geodesic complex projective spaces $\CP^{n-1}$).
\item {($A_2$)} Tubes over totally geodesic complex projective spaces
$\CP^k , 1 \le k \le n-2$.
\item {($B$)} Tubes over complex quadrics (which are also tubes over totally geodesic real projective spaces $\R P^n$).
\item {($C$)} Tubes over the Segre embedding of $\C P^1 \times \C P^m$
where $2m+1=n$ and $n \ge 5$.
\item {($D$)} Tubes over the Pl\"ucker embedding of the complex Grassmann
manifold $G_{2, 5}$ (which occur only for $n = 9$).
\item {($E$)} Tubes over the canonical embedding of the Hermitian
symmetric space $SO(10)/U(5)$ (which occur only for $n = 15$).
\end{itemize}
Note that only types $A_1$ and $B$ can occur when $n=2$.

{\bf Montiel's list for $\CH^n$}

\begin{itemize}

\item {($A_0$)} Horospheres.
\item {($A_1$)} Geodesic spheres and tubes over totally geodesic complex
hyperbolic spaces $\CH^{n-1}$.
\item {($A_2$)} Tubes over totally geodesic complex hyperbolic spaces $\CH^k , 1 \le k \le n-2$.
\item {($B$)} Tubes over totally geodesic real hyperbolic spaces $\RH^n$.

\end{itemize}
Note that Type $A_2$ cannot occur when $n=2$.

\section {The Hopf Case}
\begin{theorem}\label{theorem1}

Let $M^3$ be a Hopf hypersurface in $\CP^2$ or $\CH^2$ with $g \le 2$ distinct principal
curvatures at each point.
Then $M$ is an open subset of a hypersurface in the lists of Takagi and Montiel.
\end{theorem}
\begin{proof}
It is well-known (see Theorem 1.5 of \cite{nrsurvey}) that umbilic hypersurfaces
cannot occur in $\CP^n$ or $\CH^n$.  In fact, Hopf hypersurfaces cannot have umbilic points, since
by Lemma \ref{fundamental} the Hopf principal curvature $\a$ would have to satisfy
$\a^2 = \a^2 + c$ at such points.
Thus, when $n=2$ the multiplicity of $\a$ as a principal curvature is either 1 or 2 at each point $p\in M$, and by continuity
the multiplicity will be the same on an open set around $p$.  Hence the set of points
where $\a$ has multiplicity 2, and the set of points where $\a$ has multiplicity 1
(which coincides with the set of points where the holomorphic subspace $W^\perp$ is principal),
are both open and closed in $M$.  So, one set is empty and the other is all of $M$.

If $\a$ has multiplicity 2 on $M$, Lemma \ref{fundamental} shows that
$\a\nu = \a^2 + 2c$ where $\nu$ is the other principal curvature.  Thus, if $\a^2 + 2c \ne 0$ then $\nu$ must be a nonzero constant,
 while if $\a^2 + 2c =0$ then $\nu$ must be identically zero.
 The classification of Hopf hypersurfaces with constant principal curvatures by Kimura \cite{kimura} and Berndt \cite{Berndt}
 implies that $M$ is an open subset of a hypersurface in the Takagi/Montiel lists.
In fact, $M$ must be a Type $B$ hypersurface in $\CH^2$ (a tube around $\RH^2$) of radius $ru$ with $\coth u =  \sqrt 3$.

The other possibility is that $\a$ has multiplicity 1 on $M$.
Then the other principal curvature satisfies $\nu^2 = \a \nu + c$ and so is constant.
Again, $M$ must be an open subset of a hypersurface in the Takagi/Montiel list, in this case Type $A_0$ (a horosphere in $\CH^2$) or Type $A_1$ (a geodesic sphere in $\CP^2$ or $\CH^2$, or a tube over a totally geodesic $\CH^1$ in $\CH^2$).
\end{proof}

\section{The non-Hopf case}\label{2pc}
Consider now a hypersurface $M$ in the ambient space $\X$ (either $\CP^2$ or $\CH^2$).
If $M$ is not Hopf, then $AW \ne \alpha W$ on a nonempty open subset of $M$, and we can construct the
{\em standard frame} $(W,X,Y)$ as follows.  First, choose the unit vector field $X$
so that $A W= \alpha W + \beta X$ for a positive function $\beta$; then let $Y=\varphi X$.
Then $A$ is represented with respect to this frame by a matrix
\begin{equation}\label{shapematrix}
\begin{pmatrix} \alpha & \beta & 0 \\ \beta & \lambda & \mu \\ 0 & \mu & \nu\end{pmatrix}
\end{equation}
where $\lambda, \mu, \nu$ are also smooth functions.

\begin{prop}\label{algprop}
Let $M^3$ be a hypersurface in $\CP^2$ or $\CH^2$ and suppose that $AW \ne \alpha W$ on $M$.  Then
there are  $g \le 2$ distinct principal curvatures at each point if and only if
$\mu=0$ and
\begin{equation}\label{nucond}
\nu^2-(\alpha+\lambda)\nu + (\lambda\alpha - \beta^2) = 0.
\end{equation}
\end{prop}
\begin{proof}
Since $AW \ne \alpha W$, the setup leading to \eqref{shapematrix} holds, and therefore $g \ge 2$ globally.  Suppose now that
$g=2$ everywhere.  We will construct the standard frame $(W,X,Y)$ in a slightly different way.

First, note that $W^\perp$ intersects the two-dimensional principal space in a one-dimensional subspace.
On any simply-connected domain in $M$, let $\widetilde Y$ be a unit principal vector field lying in $W^\perp$,
 corresponding to the principal curvature $\nu$ of multiplicity 2.  Let $\widetilde X = -\varphi \widetilde Y$.
Then $(W, \widetilde X, \widetilde Y)$ is a local orthonormal frame.

Since the span of $\{W, \widetilde X\}$ is $A$-invariant, $A$ is represented by a matrix
of the form \eqref{shapematrix}, with $\mu=0$.  (Although $\beta$
was specified to be a positive function in \eqref{shapematrix}, this can easily be arranged by changing the sign of $Y$
if necessary.)  Thus, we can drop the tildes on $X$ and $Y$.

Furthermore, $\nu$
must be an eigenvalue of the upper-left $2\times 2$ submatrix of $A$, from which the formula \eqref{nucond}
follows. The converse is trivial.
\end{proof}

Proposition \ref{algprop} implies that non-Hopf hypersurfaces with $g=2$ are part of a class of hypersurfaces
previously investigated by D\'iaz-Ramos and Dom\'inguez-Vazquez \cite{DD11} in the context of constant principal curvatures.
Namely, one defines a distribution $\H$ to be the span of $\{ W, AW, A^2 W, \ldots \}$.  For each $x\in M$,
 $\H_x \subset T_x M$ is the smallest subspace that contains $W_x$ and is invariant under $A$.
 D\'iaz-Ramos and Dom\'inguez-Vazquez study hypersurfaces where $\H$ has constant rank 2.  (This
is a generalization of the Hopf condition, under which $\H$ has constant rank one.)
Since from Proposition \ref{algprop} we have
$\mu=0$, it is clear that if $M$ is non-Hopf with $g=2$ then $\H$ has rank 2 for these hypersurfaces, but more is true:

\begin{theorem}\label{g2necess}  Let $M$ be a hypersurface in $\X$ with $AW \ne \alpha W$ and
$g \le 2$ principal curvatures at each point.
Then $\H$ has rank 2, and is integrable.  Furthermore, the derivatives of
components $\alpha, \beta, \lambda,$ and $\nu$ are zero along directions tangent to $\H$, and they satisfy
\begin{equation}\label{odesys}
\begin{aligned}
\dfrac{d\alpha}{ds} &= \beta(\alpha+\lambda-3\nu)\\
\dfrac{d\beta}{ds} &= \beta^2 + \lambda^2 +\nu(\alpha-2\lambda)+c\\
\dfrac{d\lambda}{ds} &= \dfrac{ (\lambda-\nu)(\lambda^2-\alpha\lambda-c)}{\beta} + \beta(2\lambda+\nu),
\end{aligned}
\end{equation}
where $d/ds$ stands for the derivative with respect to $Y$.
\end{theorem}

We will postpone the proof of this theorem until section \ref{edssect}.

Hypersurfaces in $\CP^2$ or $\CH^2$ with $\H$ of rank 2 and integrable
are discussed in section \ref{weakly}, where we prove the following existence result:

\begin{theorem}\label{constructor}  Suppose $\alpha(s), \beta(s), \lambda(s), \nu(s)$ comprise a smooth solution of the underdetermined system
\eqref{odesys}, defined for $s$ in an open interval $I\subset \R$, and such that $\beta(s)$ is nonvanishing.
Then there exists a smooth immersion $\Phi:I\times \R^2\to \X$ determining a hypersurface $M$, equipped with
a standard frame $(W, X, Y)$, such that $\Phi$ maps the $\R^2$-factors onto leaves of $\H$.
The components of the shape operator are constant along these leaves and they coincide with the given solution.  
Furthermore, the leaves are homogeneous and have Gauss curvature zero.
\end{theorem}

\begin{corollary}  Suppose $\alpha(s), \beta(s), \lambda(s), \nu(s)$ satisfy the system
\eqref{odesys} and the algebraic condition \eqref{nucond}.  Then the hypersurface constructed
by Theorem \ref{constructor} is a non-Hopf hypersurface with two distinct principal curvatures.
Conversely, every such hypersurface is locally of this form.
\end{corollary}
\begin{proof}  The first statement follows immediately from Theorem \ref{constructor} and
the `if' part of Proposition \ref{algprop}. The second statement follows from the `only if'
part of the proposition and Theorem \ref{g2necess}.
\end{proof}

This last result shows that Theorems 4.6 and 4.7 of \cite{nrsurvey}, quoted at the beginning of section
\ref{intro}, do not extend to $n=2$.  For, given initial values $\alpha_0, \beta_0, \lambda_0$ such that
 $\beta_0\ne 0$, we can define a function
$F(\alpha,\beta,\lambda)$ on a neighborhood of this point in $\R^3$ such that $\nu=F(\alpha,\beta,\lambda)$ satisfies \eqref{nucond} identically.
Then substituting this for $\nu$ in the system \eqref{odesys} gives a determined system.  Applying standard
existence theory for systems of ODE, and using our initial values at $s=0$, will yield a solution
$\alpha(s)$, $\beta(s)$, $\lambda(s)$ of \eqref{odesys} with $\nu(s) = F(\alpha(s), \beta(s), \lambda(s))$
(so that \eqref{nucond} is satisfied), and which is defined for $s$ on an open interval $I$ containing zero.  Because the
system is autonomous, using the values $(\alpha(s_1), \beta(s_1), \lambda(s_1))$ for any nonzero $s_1 \in I$ as initial conditions
for the system will recover the same solution.

To summarize, there is a two-parameter family of solution trajectories $T$ for \eqref{odesys} which satisfy \eqref{nucond}
with $\beta$ non-vanishing.  Each of these determines a non-Hopf hypersurface
$M_T$ with $g=2$, up to rigid motions.  Conversely, given any non-Hopf hypersurface $M'$ with $g=2$ and $p_0 \in M'$,
we may use the components of the shape operator of $M'$ at $p_0$ as initial conditions to determine an $M_T$ which
is congruent to an open subset of $M'$ containing $p_0$.

\section{Moving Frames Calculations}\label{edssect}
In this section we will prove Theorem \ref{g2necess} using the techniques of moving frames
and exterior differential systems.  Background material in this subject may be found in the
textbook \cite{cfb}.  We begin by reviewing the geometric structure of the frame bundle that we will use.

An orthonormal frame $(e_1, e_2, e_3, e_4)$
at a point in $\X$ is defined to be unitary if $J e_1 = e_2$ and $J e_3 = e_4$.
We let $\F$ be the bundle of {\em unitary frames} on $\X$.
On $\F$ there are canonical forms $\w^i$ and connection forms $\w^i_j$ for $1\le i,j \le 4$.  These have the
property that if $(e_1, e_2, e_3, e_4)$ is any local unitary frame field on $\X$ and $f$ is the corresponding local section
of $\F$, then the $f^*\w^i$ comprise the dual coframe field, and
$$\langle e_i, \widetilde\nabla_\bv e_j \rangle = \bv \intprod f^* \w^i_j$$
where $\widetilde\nabla$ is the Levi-Civita connection on $\X$ and  $\bv$ is tangent to $M$.  The connection forms satisfy $\w^i_j = -\w^j_i$, but also the {\em structure equations}
$$d\w^i = -\w^i_j \wedge \w^j, \qquad d\w^i_j = - \w^i_k \wedge \w^k_j + \Omega^i_j,$$
where $\Omega^i_j$ are the curvature 2-forms.  The latter encode the curvature tensor of $\X$,
because for any local section,
$$
f^*\Omega^i_j = \tfrac12 R^i_{jk\ell} f^* (\w^k \wedge \w^\ell), \qquad R^i_{jk\ell}  = \langle e_i, R(e_k,e_\ell) e_j\rangle.
$$
(The structure equations and their relation to the curvature tensor hold on the orthonormal
frame bundle of any Riemannian manifold; see \S2.6 in \cite{cfb}.)
Moreover, because $\X$ is a K\"ahler manifold and hence $J$ is parallel with
respect to $\widetilde\nabla$, we have
$$\w^3_1 = \w^4_2, \quad \w^3_2 = -\w^4_1,
$$
with similar relationships holding among the curvature 2-forms.
We will use 1-forms $\w^1, \ldots, \w^4$, $\w^2_1$, $\w^4_1$, $\w^4_2$, $\w^4_3$ as a (globally defined) coframe
on $\F$.  In order to compute the exterior derivatives of these 1-forms, we will need to know the
curvature 2-forms.  Using the fact that $\X$ is a space of constant holomorphic sectional curvature $c$, Theorem
1.1 in \cite{nrsurvey} implies that
\begin{align*}
\Omega^2_1 &= -c(4 \w^1 \wedge \w^2 + 2 \w^3 \wedge \w^4), &
\Omega^4_1 &= -c(\w^1 \wedge \w^4 - \w^2 \wedge \w^3), \\
\Omega^4_3 &= -c(2\w^1 \wedge \w^2 + 4\w^3 \wedge \w^4), &
\Omega^4_2 &= -c(\w^1 \wedge \w^3 + \w^2 \wedge \w^4).
\end{align*}

Our method for studying and constructing (framed) hypersurfaces will be to treat the images
of the sections $f$ as integral submanifolds of a Pfaffian exterior differential system.
Briefly, a Pfaffian system $\I$ on a manifold $B$ is a graded ideal inside the
algebra $\Omega^* B$ of differential forms on $B$, which near any point is generated algebraically
by a finite set of 1-forms and their exterior derivatives.  A submanifold $N \subset B$
is an integral of $\I$ if and only if $i^* \psi=0$ for all differential forms $\psi$ in $\I$, where
$i: N \to B$ is the inclusion map.  (We will often abbreviate this by saying that
$\psi = 0$ {\em along N}.)

We will next show that a frame for a hypersurface $M$, adapted as in \S4, corresponds to an integral
submanifold of a certain Pfaffian system.
Given a standard frame $(W,X,Y)$ on $M$ satisfying $AW\ne \a W$,
 we can define a local section $f:M \to \F\vert_M$ by letting
\begin{equation}\label{WeXY}
e_3 = W, \quad e_4 = JW, \quad e_2 = X, \quad e_1 = -Y.
\end{equation}
Then the pullbacks of the 1-forms on $\F$ satisfy
\begin{equation}\label{pullsat}
f^* \w^4 =0, \quad f^*\begin{bmatrix}\w^4_3 \\  \w^4_2  \\ \w^4_1 \end{bmatrix}
= \begin{pmatrix} \alpha & \beta & 0 \\ \beta & \lambda & -\mu \\
0 & -\mu & \nu \end{pmatrix} f^*\begin{bmatrix} \w^3 \\ \w^2 \\  \w^1 \end{bmatrix}
\end{equation}
where $\alpha, \beta, \lambda,\mu, \nu$ are the functions on $M$ giving the components of
the shape operator \eqref{shapematrix}.

To see that these conditions are equivalent to the vanishing of certain 1-forms along a submanifold,
we need to introduce the shape operator components as extra variables.  In particular,
we let $\alpha, \beta, \lambda,\mu, \nu$
be coordinates on $\R^5$, and we define following 1-forms on $\F \times \R^5$:
\begin{align*}
\theta_0 &:= \w^4\\
\theta_1 &:= \w^4_1 - \nu \w^1 + \mu \w^2\\
\theta_2 &:= \w^4_2 + \mu \w^1 - \lambda \w^2 - \beta \w^3\\
\theta_3 &:= \w^4_3 - \beta \w^2 - \alpha \w^3.
\end{align*}
(The $\w^i$ and $\w^i_j$ are pulled back from $\F$ to $\F \times \R^5$.)
Given a standard frame on $M$, the fibered product of $f$ with the graphs of $\alpha, \beta, \lambda,\mu, \nu$
gives a mapping $\widehat{f}:M \to \F\times \R^5$ whose image is a 3-dimensional integral manifold
of the Pfaffian system generated by 1-forms $\theta_0, \ldots, \theta_3$.  Moreover, this submanifold
satisfies the  {\em independence condition} $\widehat{f}^* (\w^1 \wedge \w^2 \wedge\w^3) \ne 0$.
Conversely, every 3-dimensional integral manifold of the differential forms $\theta_0, \ldots, \theta_3$ that satisfies the independence condition arises in exactly this way, from a standard adapted frame along a hypersurface in $\X$.

\begin{proof}[Proof of Theorem \ref{g2necess}]
Let $\U \subset \R^5$ be defined by $\beta\ne 0$, $\mu=0$ and
\begin{equation}\label{altnucond}
(\nu-\alpha)(\nu-\lambda)=\beta^2.
\end{equation}
(This is just \eqref{nucond} rewritten.)
Because $\beta\ne0$, we have $\alpha -\nu \ne 0$ at each point of $\U$, and
$\U$ is a smooth 3-dimensional submanifold of $\R^5$.
We will use $\beta$, $\nu$ and $\tau=(\nu-\alpha)/\beta$ as coordinates on $\U$, in terms of which
\begin{equation}\label{parsubs}
\alpha=\nu -\beta\tau, \qquad \lambda = \nu - \beta/\tau.
\end{equation}
Note that $\tau$ is always nonzero.
The geometric meaning of $\tau$ is that, if we write the unit $\nu$-eigenvector in
the span of $\{W,X\}$ as $\cos\phi\, W + \sin\phi\, X$, then $\tau=\tan\phi$.

Now let $M$ be a non-Hopf hypersurface with $g \le 2$ distinct principal curvatures.
In the proof of Proposition \ref{algprop} we developed a (local) standard frame on $M$ for which
the components of the shape operator satisfy $\mu=0$ and \eqref{altnucond}.
Let $f$ be the corresponding section of $\F\vert_M$.
Then the image of $\widehat{f}$ is an integral of $\theta_0, \ldots, \theta_3$ which
lies in $\F\times \U$.
Accordingly, we pull back the forms  from $\F \times \R^5$ to $\F\times \U$, giving
\begin{align*}
\theta_1 &= \w^4_1 - \nu \w^1\\
\theta_2 &= \w^4_2 - ( \nu -\beta/\tau) \w^2 - \beta \w^3\\
\theta_3 &= \w^4_3 - \beta \w^2 - (\nu-\beta\tau) \w^3,
\end{align*}
and let $\I$ be the Pfaffian exterior differential system on $\F\times \U$ generated by these re-defined
1-forms.

As an algebraic ideal, $\I$ is generated by these 1-forms and their exterior derivatives.  We may
simplify the latter by omitting wedge products involving the $\theta_0, \ldots, \theta_3$ as factors.  For example, we compute
$$-d\theta_0 = \theta_1 \wedge \w^1 +\theta_2 \wedge \w^2 +\theta_3 \wedge \w^3,$$
so that $d\theta_0$ adds no new algebraic generators for the ideal; we express this fact by writing
$d\theta_0 \equiv 0$ mod $\theta_0, \ldots, \theta_3$.  Similarly, we compute
\begin{equation}\label{twofer}
\left.
\begin{aligned}
-d\theta_1 &\equiv \pi_1 \& \w^1+ \pi_4 \& (\w^2-\tau\w^3),\\
-d\theta_2 &\equiv \pi_4 \&\w^1 + \pi_3 \& \w^2 + \pi_2 \& \w^3,\\
-d\theta_3 &\equiv -\tau\pi_4 \& \w^1 +\pi_2 \& \w^2 +\left((1+\tau^2)\pi_1 -2\tau \pi_2 -\tau^2 \pi_3\right) \& \w^3
\end{aligned}\right\} \mod \ \theta_0, \ldots, \theta_3,
\end{equation}
where
\begin{equation}\label{pidefs}
\begin{aligned}
\pi_1 &:= d\nu + \dfrac{3c\tau}{1+\tau^2}\w^1, \\
\pi_2 &:= d\beta + \left(\beta^2-2\beta\nu\tau +2c +\dfrac{\beta \nu}{\tau}\right)\w^1, \\
\pi_3 &:= d\left(\nu - \dfrac{\beta}\tau\right)+\left(3\beta \nu - \dfrac{\beta^2}{\tau}\right) \w^1,\\
\pi_4 &:= \dfrac{\beta}{\tau}\w^2_1 +(c-\beta\nu\tau)\w^3.
\end{aligned}
\end{equation}
The 1-forms $\pi_1, \ldots, \pi_4$, along with $\w^1,\w^2,\w^3$ and $\theta_0, \ldots, \theta_3$,
complete a coframe on $\F\times \U$ which is adapted to $\I$ in the sense that the generator
2-forms of $\I$ are most simply expressed in terms of this coframe.

Suppose that $\Sigma$ is an integral 3-fold of $\I$ satisfying the independence condition.
Let $\Theta_1, \Theta_2, \Theta_3$ be the 2-forms on the right-hand side of \eqref{twofer}, 
which must vanish along $\Sigma$.  The vanishing of $\Theta_1$ implies (using Cartan's Lemma) that
\begin{equation}\label{firstpie} \pi_1 = m\, \w^1 + p\,\wt^2, \qquad \pi_4 = p\, \w^1 + q\, \wt^2
\end{equation}
for some functions $m,p,q$ along $\Sigma$.  (For convenience, we will let $\wt^2$ denote $\w^2 - \tau \w^3$
from now on.)
On the other hand,
$$\Theta_3 + \tau \Theta_2 = (\pi_2 + \tau \pi_3) \wedge \wt^2 + (1+\tau^2)\pi_1 \wedge \w^3,$$
and applying Cartan's Lemma to the vanishing of this 2-form yields
\begin{equation}\label{secondpie}
\pi_2 + \tau \pi_3 = s\,\wt^2 + u\, \w^3, \qquad  (1+\tau^2) \pi_1 = u\, \wt^2 +  v\, \w^3
\end{equation}
for some functions $s,u,v$ along $\Sigma$.  Comparing \eqref{firstpie} and \eqref{secondpie}
shows that $u = p(1+\tau^2)$ and $m=v=0$; hence
$$\pi_1 = p \,\wt^2, \quad \pi_4 = p \,\w^1 + q\,\wt^2, \quad \pi_2 + \tau\pi_3 = p (1+\tau^2) \w^3 + s\, \wt^2$$
along $\Sigma$.  Substituting these into the equation $\Theta_2=0$ implies that
$$\pi_3 = q\, \w^1 + t\, \wt^2 + s\, \w^3 $$
for an additional function $t$ along $\Sigma$.

Substituting these values into the definitions \eqref{pidefs} of $\pi_1$ through $\pi_4$ lets us determine
the values of the exterior derivatives
\begin{equation}\label{dnbl}
\begin{aligned}
d\nu &= - \dfrac{3 c \tau}{1+\tau^2} \w^1+ p\,\wt^2, \\
d\beta &= \left(2\beta \nu \tau - \beta^2 - \dfrac{\beta\nu}{\tau} - 2 c -\tau q\right) \w^1 + (s-\tau t) \wt^2 + (p(1+\tau^2) - \tau s) \w^3,\\
d\lambda &= d\left(\nu - \dfrac{\beta}\tau\right) = \left( \dfrac{\beta^2}{\tau} - 3\beta\nu +q\right) \w^1 + t\, \wt^2 + s\,\w^3
\end{aligned}
\end{equation}
along $\Sigma$, as well as
$$\w^2_1 = \dfrac{\tau}{\beta}( p\,\w^1 + q\,\wt^2 + (\beta \nu \tau - c) \w^3).$$
Of course, the 1-forms on the right in \eqref{dnbl} must be closed along $\Sigma$.  Computing the
exterior derivatives of these, modulo $\theta_0, \ldots, \theta_3$ and using the above values for $d\nu, d\beta, d\lambda$
and $\w^2_1$, gives algebraic conditions that $p,q,s,t$ must satisfy.  For example, we compute that
$$ 0 = d^2 \nu \wedge \wt^2 = \dfrac{4c\tau(2\tau^2-1)}{\beta(1+\tau^2)}  p\, \w^1 \wedge \w^2 \wedge \w^3$$
along $\Sigma$.  The vanishing of this 3-form implies that at each point of $\Sigma$,
either $p=0$ or $\tau^2 = \tfrac12$.  So, if $p\ne 0$ at some point of $\Sigma$, then 
$\tau^2 = \tfrac12$ on an open set around that point.  On the other hand, the equations \eqref{dnbl} give
\begin{equation}\label{deetau}
d\tau = - \left( \nu(1+\tau^2) + \dfrac{c\tau(2-\tau^2)}{\beta(1+\tau^2)}\right)\w^1 -\dfrac{\tau(p\tau-s)}{\beta}\wt^2 + \dfrac{p \tau(1+\tau^2)}{\beta}\w^3,
\end{equation}
so that if $\tau$ is constant on an open set in $\Sigma$ then $p=0$ on that
set.  Thus, we conclude that $p$ vanishes identically on $\Sigma$.

With this conclusion taken into account, we have
$$ 0 = d^2\nu = \dfrac{3c(1-\tau^2)}{\beta(1+\tau^2)^2}s\, \w^1 \wedge \wt^2
+ \dfrac{3c}{\beta(1+\tau^2)}\left( (\beta\nu-q) \tau^3 - c\tau^2 +\beta^2 -\beta\nu\tau\right) \w^2 \wedge \w^3$$
along $\Sigma$.  From the first term, at each point either $s=0$ or $\tau^2 = 1$, and \eqref{deetau} (with $p=0$) lets us conclude that $s=0$ whenever $\tau$ is locally constant.
Hence $s$ vanishes identically on $\Sigma$, and from the second term
$$q = \dfrac{\beta \nu \tau^3 - c\tau^2 +\beta^2 -\beta\nu\tau}{\tau^3}.$$
With these values for $p,q,s$ taken into account, we have
$$0 = d^2\tau = \dfrac{c \tau^2 (2-\tau^2)}{\beta^2(1+\tau^2)} t\,\w^1 \wedge \wt^2$$
along $\Sigma$, which implies that at each point either $t=0$ or $\tau^2 = 2$.  If the latter happens on an open
set, then \eqref{deetau} implies that $\nu=0$ on that same open set; in that case, substituting $\nu=0$ and $d\tau=0$ into the
system 2-forms gives
$$\Theta_3+\tau \Theta_2 = -c(2\w^2 + \tau\w^3) \wedge \w^1,$$
so that the vanishing of these 2-forms is incompatible with the independence condition.
From this contradiction we conclude that $t=0$ identically on $\Sigma$.

Using the values deduced for $p,q,s,t$ above, we conclude
that the following 1-forms vanish along $\Sigma$:
\begin{align*}
\theta_4 &:= d\nu + \dfrac{3c\tau}{1+\tau^2}\w^1, \\
\theta_5 &:= d\beta + \left(\beta^2+\dfrac{\beta^2}{\tau^2}-\beta\nu\tau +c\right)\w^1, \\
\theta_6 &:= d\lambda+\dfrac{c\tau^2 - \beta^2(1+\tau^2) + 2\beta\nu\tau^3 + \beta\nu\tau}{\tau^3} \w^1,\\
\theta_7 &:= \dfrac{\beta}{\tau}\w^2_1 +\dfrac{c\tau^2 -\beta^2 -\beta\nu\tau^3+\beta\nu\tau}{\tau^3}\w^2
+\dfrac{\beta^2 - \beta\nu\tau}{\tau^2}\w^3.
\end{align*}

Now we can establish the assertions in the theorem.  Note first that $\widehat{f}^* \w^1$ annihilates
the distribution $\H$ on $M$, so to show that $\H$ is integrable we compute
\begin{align*}d\w^1 &\equiv -\w^2_1 \wedge \w^2 - \w^4_1 \wedge \w^3 \mod \theta_0\\
                    &= \lambda \w^2 \wedge \w^3 - \lambda \w^2 \wedge \w^3 \mod \theta_7\\
                    &=0.
\end{align*}
Next, the vanishing of $\theta_4, \theta_5, \theta_6$ implies that
$\nu,\beta,\lambda$ (and hence $\alpha$) are constant along the leaves of $\H$.  Since $Y=-e_1$, the vanishing
of these 1-forms also gives us the $Y$-derivatives of these variables:
\begin{equation}\label{yelp}
\begin{aligned}
Y\nu &= \dfrac{3c\tau}{1+\tau^2},\\
Y\beta &= \beta^2+\dfrac{\beta^2}{\tau^2}-\beta\nu\tau +c  =\beta^2+ (\lambda-\nu)^2+\nu(\alpha-\nu)+c,\\
Y\lambda = Y(\nu-\beta/\tau)&= \dfrac{c\tau^2 - \beta^2(1+\tau^2) + 2\beta\nu\tau^3 + \beta\nu\tau}{\tau^3}.
\end{aligned}
\end{equation}
Using $\alpha = \nu-\beta\tau$ we also get
$$Y\alpha = -\beta(\nu + \beta\tau + \beta/\tau) =-\beta(3\nu-\alpha-\lambda).$$
Thus,  $\alpha, \beta, \lambda$ and $\nu$ satisfy the underdetermined system of differential equations
\eqref{odesys}.
(We leave it to the interested reader to check that the right-hand side of $Y\lambda$ in \eqref{yelp} coincides
with the third equation in \eqref{odesys}, once the substitutions \eqref{parsubs} are made.)
\end{proof}

\begin{remark}
In the preceding proof, the possible values of the 1-forms
$\w^2_1, d\beta, d\nu, d\tau$ (as well as $\w^4, \w^4_1, \w^4_2, \w^4_3$) on
an integral submanifold $\Sigma$ satisfying the independence condition are completely determined
at each point of $\F\times \U$.  In other words, $\Sigma$ must be tangent
to the rank 3 distribution on $\F\times \U$ whose tangent spaces are annihilated by
$\theta_0, \ldots, \theta_7$ at each point.  Of course, 3-dimensional integral submanifolds
of such a distribution only exist at points where the Frobenius condition is satisfied.
This means we must check that the exterior derivatives of $\theta_0, \ldots, \theta_7$ are
zero modulo these same 1-forms.  Fortunately, this condition holds identically, and there exists
a unique local integral submanifold $\Sigma$ through each point of $\F\times \U$.  Global existence
will follow from Theorem \ref{constructor}.
\end{remark}

\section{2-Hopf Hypersurfaces}\label{weakly}

As remarked earlier, one can view the condition that $\H$ is of rank 2 as a weakening of the Hopf hypersurface condition.
By itself, this condition is too weak:  for example, one can show using exterior differential systems that 3-dimensional hypersurfaces in $\X$ for which $\H$ has rank 2 are abundant, at least locally, since examples can be constructed using the Cartan-K\"ahler theory
depending on a choice of two functions of two variables.  Thus, we impose the additional condition that
$\H$ is an integrable distribution.  Such hypersurfaces are still quite flexible, and depend locally on a choice
of five functions of one variable; for example, there exists such a hypersurface through any given curve $\Gamma$
in $\X$, with the distribution $\H$ prescribed along $\Gamma$ (provided that the prescribed 2-plane is
neither tangent nor perpendicular to the curve at any point).  Motivated by these considerations, we make the following definition:
\defn{A hypersurface in $\CP^n$ or $\CH^n$ is said to be $k$-Hopf if $\H$ is integrable and of rank $k$.}

When $k=1$, the integrability condition is vacuous and this reduces to the usual notion of a Hopf hypersurface.  Also note that if $M^{2n-1}$ is
$k$-Hopf, then $1 \le k \le 2n-1$.  Because of Theorem \ref{g2necess}, we are interested in the 2-Hopf condition when $n=2$.

\begin{prop}\label{flat} Let $M^3 \subset \X$ be a 2-Hopf hypersurface.  Then the $\H$-leaves
within $M$ are flat.
\end{prop}
\begin{proof} Let $(W,X,Y)$ be a standard (local) frame on $M$, and let $f: M \to \F$ be the
associated unitary frame, as defined by \eqref{WeXY}.  With respect to the $(W,X,Y)$ basis,
the shape operator has the form \eqref{shapematrix} with $\mu=0$.  So,
$f^* \w^4_1 = \nu f^*\w^1$.  By hypothesis, $f^*\w^1$ is integrable (i.e., its exterior derivative
is zero modulo itself), so the same is true of $f^*\w^4_1$.  Because
$\w^2$ and $\w^3$ restrict to be an orthonormal coframe along an $\H$-leaf, we compute the
Gauss curvature $K$ using the equation
$$d\w^2_3 \equiv K \w^2 \wedge \w^3 \quad\mod \ \w^1.$$
(All forms here are understood to be pulled back via $f$.)  Since $d\w^2_3 = -d\w^4_1 \equiv 0$ modulo $\w^1$,
then $K=0$.
\end{proof}

\begin{remark}
It is easy to see that the 2-Hopf condition on $M^3$ requires that $\nabla_W X = \lambda Y$.  In fact, for a hypersurface such that $\H$ has rank 2, this is equivalent to integrability of $\H$.
To explain this, note that $\nabla_X W = \varphi A X=  \lambda Y$.  Then
$$
\langle \nabla_W X, W\rangle = - \langle X, \nabla_W W\rangle\\
= - \langle X, \varphi A W\rangle \\
= 0.
$$
Then $[X, W] = \nabla_X W - \nabla_W X = \lambda Y - \langle\nabla_W X, Y\rangle Y$.  For integrability,
$\langle [X, W], Y\rangle$ must be zero.  In addition, integrability implies that $\H$-components of
 $\nabla_W X$ and $\nabla_W W$ (as well as those of $\nabla_X X$ and $\nabla_X W$) vanish,
 which explains why the curvature tensor of each leaf must vanish.
\end{remark}

\bigskip

We now turn to the more specialized hypersurfaces of Theorem \ref{constructor}, which can
be characterized as follows:

\begin{prop}\label{charp} Let $M^3$ be as in Proposition \ref{flat}.  If
$\alpha =\langle AW, W\rangle$ is constant along the $\H$-leaves, then all the other
components of the shape operator (with respect to a standard basis) are also constant along these leaves,
and satisfy the differential equations \eqref{odesys}.
\end{prop}
\begin{proof}
As in section \ref{edssect}, we will set up a Pfaffian exterior differential system whose solutions are
framed hypersurfaces of the type under consideration.
Here, the system will encode the conditions that $\H$ is rank 2 and integrable.  If $f:M \to \F$ is a unitary
frame derived from a standard basis as in \eqref{WeXY}, then the pullbacks of the 1-forms on $\F$ satisfy
\eqref{pullsat} with $\mu=0$.  So, we define 1-forms
\begin{align*}
\theta_0 &:= \w^4\\
\theta_1 &:= \w^4_1 - \nu \w^1\\
\theta_2 &:= \w^4_2 - \lambda \w^2 - \beta \w^3\\
\theta_3 &:= \w^4_3 - \beta \w^2 - \alpha \w^3,
\end{align*}
with $\alpha, \beta,\lambda,\nu$ as extra variables.  

Since $\H$ is annihilated by the pullback of $\w^1$,   it  is integrable 
if and only if $d\w^1$ is a multiple of $\w^1$. 
We compute
\begin{equation}\label{dwi}
d\w^1 \equiv (\w^2_1-\lambda\w^3) \wedge \w^2 \quad \mod \theta_0, \theta_1, \theta_2, \theta_3.
\end{equation}
Thus, $\H$ is integrable 
if and only if $\w^2_1 - \lambda \w^3$ equals some linear combination of $\w^1$ and $\w^2$. 
This is equivalent to the vanishing of 
$$\theta_4:=\w^2_1 - \rho \w^1 -\delta\w^2 -\lambda \w^3,$$
for some functions  $\delta$ and $\rho$ along $M$.   (As with the shape operator components, we 
will introduce $\delta$ and $\rho$  as new variables.)  
We compute
$$d\theta_1 \equiv (\beta\delta - \beta^2-\lambda^2 + \alpha\lambda +c) \w^2 \wedge \w^3 \quad \mod \theta_0, \ldots, \theta_4, \w^1,$$
and thus we must have
\begin{equation}\label{qval}
\delta = \dfrac{\beta^2+\lambda^2 - \alpha\lambda -c}{\beta}
\end{equation}
in order to satisfy the independence condition.

Accordingly, let $\W \subset \R^5$ be the open set with coordinates $\alpha, \beta,\lambda, \nu$ and $\rho$, with
$\beta\ne 0$, and let $\I$ be the Pfaffian system  on $\F\times \W$  generated by $\theta_0, \ldots, \theta_4$ (with $\delta$ given by \eqref{qval}).  Differentiating these 1-forms modulo their span gives the generator 2-forms
\begin{equation}\label{h2tableau}
\left.
\begin{aligned}
-d\theta_1 &\equiv \pi_1 \wedge \w^1, \\
-d\theta_2 &\equiv \pi_2 \wedge \w^2 + \pi_3 \wedge \w^3,\\
-d\theta_3 &\equiv \pi_3 \wedge \w^2 + \pi_4 \wedge \w^3,\\
-d\theta_4 &\equiv \pi_5 \wedge \w^1+\left( \dfrac{2\lambda-\alpha}\beta \pi_2 + \dfrac{\beta^2 -\lambda^2 +\alpha \lambda+c}{\beta^2}\pi_3 - \dfrac{\lambda}\beta \pi_4\right) \wedge \w^2
+\pi_2 \wedge \w^3.
\end{aligned}
\right\}
\mod \theta_0, \ldots, \theta_4,
\end{equation}
where
\begin{align*}
\pi_1 &= d\nu + \rho(\lambda-\nu)\w^2 + \beta\rho\, \w^3,\\
\pi_2 &= d\lambda +\left( \dfrac{ (\lambda-\nu)(\lambda^2-\alpha\lambda-c)}{\beta} + \beta(2\lambda+\nu)\right)\w^1,\\
\pi_3 &= d\beta + (\beta^2 +\lambda^2 + \nu(\alpha -2\lambda) + c)\w^1, \\
\pi_4 &= d\alpha + \beta(\alpha+\lambda-3\nu)\w^1,\\
\pi_5 &= d\rho - \rho^2\, \w^2.
\end{align*}

\begin{remark}
Integral submanifolds of $\I$ are in 1-to-1 correspondence with 2-Hopf hypersurfaces in $\X$, so it is of interest
to know how large the set of such surfaces is.  The system $\I$ is not involutive.  However,
it is easy to see from the 2-form generators \eqref{h2tableau} that $\pi_5 \wedge \w^1$ must vanish along
any integral manifold satisfying the independence condition.
When this 2-form is adjoined, the resulting ideal is involutive,
and Cartan's Test indicates that solutions depend on five functions of one variable.
For example, given any curve $\gamma$ in $\X$ and a 2-plane field $E$ along $\gamma$
which is transverse to the $J$-invariant subspace containing $T_p\gamma$ for every $p \in \gamma$,
there is a 2-Hopf hypersurface containing $\gamma$ with $\H = E$ along $\gamma$.
\end{remark}

The set of 2-Hopf hypersurfaces $M$ satisfying the additional hypothesis that $\alpha$ is constant
along $\H$-leaves is considerably smaller.
For, let $\Sigma$ be the integral manifold of $\I$ corresponding to such a hypersurface.
Because $\H$ is annihilated by the pullbacks to $M$ of  $\w^2$ and $\w^3$, then
the constancy of $\alpha$ along the $\H$-leaves implies that $\pi_4$ must be
a multiple of $\w^1$ along $\Sigma$.  On the other hand, the
vanishing of the 2-form $d\theta_3$ implies, by the Cartan Lemma, that
$\pi_4$ must be a linear combination of $\w^2$ and $\w^3$.  Thus, $\pi_4=0$ and $\pi_3 \wedge \w^2=0$ along $\Sigma$.

Substituting these into $d\theta_4$ and applying the Cartan Lemma implies
that $\pi_5$ and $\pi_2$ must be linear combinations of $\w^1$ and $\w^3 + ((2\lambda-\alpha)/\beta)\,\w^2$ along $\Sigma$.
When we compare this with result of the Cartan Lemma applied to the vanishing of $d\theta_2$, we see that there are scalars
$r,t$ such that
$$\pi_2 = r(\w^3 + ((2\lambda-\alpha)/\beta)\w^2), \quad
\pi_3 = r \w^2, \quad \pi_5 = t \w^1$$
along $\Sigma$.

Because $\pi_4$ vanishes along $\Sigma$, the same is true of its exterior derivative.  We compute
$$d\pi_4 \wedge \w^2 \equiv \beta(-3\pi_1+\pi_2 +\pi_4 +3\beta\rho\, \w^3) \wedge \w^1 \wedge \w^2 +(\alpha+\lambda-3\nu)\pi_3 \wedge \w^1 \wedge \w^2
\mod \theta_0, \ldots, \theta_4.$$
From the top line of \eqref{h2tableau}, we see that $\pi_1 \wedge \w^1=0$ along $\Sigma$; substituting
this and the values for $\pi_2, \pi_3, \pi_5$ along $\Sigma$ into the above 3-form, we conclude
that $r=-3\beta\rho$.  Hence, the 1-form
$$\pi_2 +3\rho (\beta\w^3 + (2\lambda-\alpha)\w^2)$$
must vanish along $\Sigma$.  Taking an exterior derivative of this form modulo $\theta_0, \ldots, \theta_4$,
wedging with $\w^1$,
and substituting the known values for the $\pi$'s, we conclude that $\rho=0$ identically.

Because $\pi_2$ and $\pi_3$ must vanish along $\Sigma$, and $\pi_1$ must be a multiple of $\w^1$, we respectively conclude
that $\lambda, \beta$ and $\nu$  are constant along the $\H$-leaves.  Comparing the form of $\pi_2$
and $\pi_3$ with \eqref{odesys}, we can verify that $\alpha, \beta,\lambda$ satisfy the correct differential equations.
\end{proof}

\begin{proof}[Proof of Theorem \ref{constructor}]
Given a solution $\alpha(s), \beta(s),\lambda(s),\nu(s)$ of the system \eqref{odesys} defined on interval $I$, we define a Pfaffian system
$\J$ on $I \times \F$ which is generated (in part) by
substituting each solution component for the corresponding variable in the generator 1-forms of
the system $\I$ used in the proof of Proposition \ref{charp}
(with the values for $\delta$ and $\rho$ deduced there):
$$
\begin{aligned}
\vartheta_0 &:= \w^4\\
\vartheta_1 &:= \w^4_1 - \nu(s) \w^1\\
\vartheta_2 &:= \w^4_2 - \lambda(s) \w^2 - \beta(s) \w^3\\
\vartheta_3 &:= \w^4_3 - \beta(s) \w^2 - \alpha(s) \w^3,\\
\vartheta_4 &:= \w^2_1 - \delta(s)\w^2 - \lambda(s)\w^3,\\
\end{aligned}\qquad \delta(s):= \dfrac{\beta(s)^2+\lambda(s)^2 - \alpha(s)\lambda(s) -c}{\beta(s)}.
$$
To these, we add one more generator 1-form
$$\vartheta_5 := \w^1 + ds,$$
the vanishing of which ensures that $d/ds$ is the $Y$-derivative.

Our system $\J$ is a Frobenius system, so again we might invoke the Frobenius theorem (see, e.g., \cite{Warner}) to obtain
a unique connected maximal integral 3-fold of $\J$ through any point.  Although the image of this under the
fibration $\F \to M$ would be one of the desired hypersurfaces, we would not get any information about
whether the hypersurface is complete or compact.  In what follows, we give a more explicit construction.

Recall that $\X = Q/S^1$, where $Q \subset \C^3$ is the sphere $S^5(r)$ or anti-de Sitter space $H^5_1(r)$
defined by
$$\langle \bz, \bz \rangle = \epsilon r^2,$$
and $\langle\ , \rangle$ is the hermitian inner product on $\C^3$ with signature $+++$ (when
we take $\epsilon=1$ and get $\X = \CP^2$), or $-++$ (when we take $\epsilon=-1$ and get $\X = \CH^2$).
The $S^1$ action multiplies the coordinates by a unit modulus scalar, so that the
quotient map $\pi:Q \to \X$ is just the restriction of complex projectivization.  The metric
on the orthogonal complement to the fibers of $\pi$ descends to the quotient $\X$,
and has constant holomorphic sectional curvature $4c$ where $c=\epsilon/r^2$.
For more detail, see \cite{nrsurvey}, pp. 235--237.

Recall from
\cite{dalembert} that we may lift the $S^1$-quotient to a quotient map $\Pi:G \to \F$, where
$G$ is the group of matrices that are unitary with respect to the inner product
(i.e., $G=U(3)$ for $\epsilon=1$ and $G=U(1,2)$ for $\epsilon=-1$).
In more detail, we define the submersion $\Pi:G\to \F$ as follows:
given an element $g\in G$ with columns $(E_0, E_1, E_2)$, $\Pi$ takes
$g$ to the unitary frame
$$e_1 = \pi_* E_1, \quad e_2 = \pi_* \ri E_1, \quad e_3 =\pi_* E_2, \quad e_4=\pi_* \ri E_2$$
at basepoint $\pi(E_0)$.
In fact, we can identify $G$ itself with the unitary frame bundle of $Q$, using $\bz = r E_0$ as the
basepoint map and $E_1, E_2$ generating the frame as above.
Then the $S^1$ action on $G$ that is equivariant with respect to this basepoint map
is simply left-multiplication by elements of the 1-parameter subgroup $\exp(\ri t I)$.

We will construct solutions to the system $\J$ by pulling the system back via $\Pi$ to $I\times G$,
constructing solutions there, and projecting back down.  For this purpose we need to express the
pullbacks of 1-forms $\w^i$ and $\w^i_j$ in terms of the Maurer-Cartan forms of $G$.  The latter are
complex-valued 1-forms $\psi^a_b$ defined by
$$d E_a = E_b \psi^b_a, \qquad 0 \le a,b \le 2,$$
where we regard the columns $E_a$ as vector-valued functions on $G$.
These columns satisfy
$$\langle E_0, E_0 \rangle = \epsilon,\quad \langle E_0, E_i\rangle=0,\quad \langle E_i, E_j \rangle = \delta_{ij},
\qquad 1 \le i,j \le 2,$$
so we have
$$\psi^b_a = \begin{cases}
-\epsilon\overline{\psi^a_b} & \text{if exactly one of $a,b$ is equal to zero,}\\
-\overline{\psi^a_b} & \text{otherwise.}
\end{cases}$$
Then, using formulas developed in \cite{dalembert}, we have
\begin{equation}\label{pully}
\begin{aligned}
\Pi^*(\w^1 + \ri \w^2) &= \psi^1_0, & \quad \Pi^*(\ri \w^2_1) &=\psi^1_1 - \psi^0_0,\\
\Pi^*(\w^3 + \ri \w^4) &= \psi^2_0, & \Pi^*(\ri \w^4_3) &= \psi^2_2 - \psi^0_0, \\
\Pi^*(\w^3_1 + \ri\w^4_1) &= \psi^2_1. & \\
\end{aligned}
\end{equation}
For the sake of brevity, we will write $\eta_1 :=\Re\psi^1_0$, $\eta_2 :=\Im\psi^1_0$ and $\eta_3 = \Re\psi^2_0$ in what follows.

Let $\Jhat$ be the Frobenius system on $I\times G$ generated by the pullbacks of the 1-forms of $\J$.   Using \eqref{pully},
we see that $\Jhat$ is generated by
\begin{align*}
\Pi^*\vartheta_0 &= \Im \psi^2_0,\\
\Pi^*\vartheta_1 &= \Im \psi^2_1 - \nu(s)\eta_1,\\
\Pi^*\vartheta_2 &= \Re\psi^2_1 - \lambda(s)\eta_2-\beta(s)\eta_3,\\
\Pi^*\vartheta_3 &= \Im(\psi^2_2-\psi^0_0) - \beta(s)\eta_2 - \alpha(s)\eta_3,\\
\Pi^*\vartheta_4 &= \Im(\psi^1_1-\psi^0_0) - \delta(s)\eta_2 - \lambda(s)\eta_3,\\
\Pi^*\vartheta_5 &= \eta_1 + ds.
\end{align*}
Since this system is of rank 6 on the 10-dimensional manifold $I\times G$, its
maximal integral manifolds are 4-dimensional; in particular, they are foliated by orbits of the $S^1$-action.
We specify 3-dimensional slices to this action by adding the extra 1-form
$\Im(\psi^0_0 + \psi^1_1 +\psi^2_2)$, whose exterior derivative is zero modulo the 1-forms of $\Jhat$.  (In fact, this 1-form
is closed on $G$, and its integral manifolds are the cosets of the subgroup of special unitary matrices.)  Let
$\Khat$ denote the resulting rank 7 Frobenius system.

We will obtain integral 3-manifolds of $\Khat$ as follows.
We will first construct a curve which is an integral of $\Khat$ but along which, in addition, $\eta_2=\eta_3=0$.
Then we will take a union of integral surfaces of $\Khat$ transverse to the curve.  (The Frobenius condition
guarantees that this union is an integral 3-fold of $\Khat$.)
Fix a value $s_0 \in I$ and a point $u_0\in G$; then
the curve takes the form $(s,U(s))$ where $U(s)$ is a $G$-valued matrix satisfying the initial value problem
$$\dfrac{dU}{ds} = -U \begin{pmatrix} 0 & \epsilon & 0\\ 1 & 0 & \ri \nu(s)\\ 0 & \ri \nu(s) & 0\end{pmatrix}, \qquad U(s_0) = u_0.$$
This is a system of linear ODE, so its solution is smooth and defined for all $s\in I$.

Next, for each fixed $s\in I$ we construct an integral surface of $\Khat$ passing through $U(s)$.  Along such a surface,
$\eta_2$ and $\eta_3$ are closed, so there are local coordinates $x,w$ such that $\eta_2 = dx$ and $\eta_3=dw$.  Then the surface
must be given by a $G$-valued function $V$ satisfying the simultaneous initial value problems
$$\dfrac{\partial V}{\partial x} = V A(s), \qquad \dfrac{\partial V}{\partial w} = V B(s), \qquad V(0,0) =  U(s),$$
where
\begin{align*} A(s) &= \begin{pmatrix} -\tfrac{\ri}3(\beta+\delta)& \epsilon\ri & 0 \\ \ri & \tfrac{\ri}3(2\delta(s)-\beta(s)) & -\lambda(s)\\
0 & \lambda(s) & \tfrac{\ri}3(2\beta(s)- \delta(s))\end{pmatrix}, \\
 B(s) &= \begin{pmatrix} -\tfrac{\ri}3(\alpha+\lambda)& 0 & -\epsilon \\ 0 & \tfrac{\ri}3(2\lambda(s)-\alpha(s)) & -\beta(s)\\
1 & \beta(s) & \tfrac{\ri}3(2\alpha(s)- \lambda(s))\end{pmatrix}.
\end{align*}
It is easy to check that $A$ and $B$ satisfy the solvability condition $[A(s),B(s)]=0$.
Thus,  $V$ is defined for all $x,w \in \R^2$; in fact the solution is given by
$$V = U(s) \exp( x A(s) + w B(s)),$$
and thus for a fixed $s$, $V$ sweeps out a left coset of a 2-dimensional abelian subgroup of $G$.
By varying $s$, we obtain a smooth matrix-valued function $V(s,x,w)$ such that
$(s,V(s,x,w))$ is an integral 3-fold of $\Khat$.  Then we set $\Phi(s,x,w) = b\circ \Pi( V(s,x,w))$, where
$b:\F \to \X$ is the basepoint map, to obtain the desired immersion.  Moreover, for fixed $s$ each
leaf $\Phi(s,x,w)$ is the orbit of a 2-dimensional abelian subgroup of $G$, acting by isometries on $\X$.
\end{proof}

\end{document}